\documentclass[12pt,a4paper]{article}
\usepackage[T1]{fontenc}
\usepackage[cp1250]{inputenc}
\usepackage{amsmath}
\usepackage{amssymb}
\usepackage{esint}

\makeatletter

\usepackage{amsfonts}\usepackage{amsthm}\usepackage[all]{xy}

\newcommand{\set}[1]{\left\{#1\right\}}

\newcommand{\Nat}{\mathbb{N}}

\newcommand{\CB}{\mathcal{B}}

\newcommand{\CM}{\mathcal{M}}

\newcommand{\CP}{\mathcal{P}}
\newcommand{\CQ}{\mathcal{Q}}

\newcommand{\CT}{\mathcal{T}}

\newcommand{\sd}{\bigtriangleup}
\newcommand{\eps}{\varepsilon}

\DeclareMathOperator{\ex}{\mathsf{ex}}
\DeclareMathOperator{\supp}{\mathsf{supp}}
\DeclareMathOperator{\rank}{\mathsf{rank}}

\theoremstyle{plain}
\newtheorem{thm}{Theorem}[section]

\newtheorem{lem}[thm]{Lemma}

\theoremstyle{definition}
\newtheorem{defn}[thm]{Definition}
\theoremstyle{remark}
\newtheorem{rem}[thm]{Remark}

\numberwithin{equation}{section}

\title{Rank as a function of measure}
\author{Tomasz Downarowicz, Yonatan Gutman, Dawid Huczek\footnote{This author's research is supported by the PhD grant MENII N N201 411339, Poland.}}

\makeatother

\begin{document}
\maketitle \abstract{We establish certain topological
properties of rank understood as a function on the set of invariant
measures on a topological dynamical system. To be exact, we show that
rank is of Young class LU (i.e., it is the limit of an increasing
sequence of upper semicontinuous functions).}

\section{Introduction}

Let $(X,T)$ be a topological dynamical system, where $X$ is a compact
metric space and $T:X\to X$ is a continuous mapping. Denote by $\CM_{T}(X)$
the set of all $T$-invariant Borel probability measures on $X$.
It is well known that this is a nonempty, convex compact metric set
(in the weak-star topology), whose extreme points are precisely the
ergodic measures (the collection of which we denote by $\ex\CM_{T}(X)$).
Moreover, $\CM_{T}(X)$ is in fact a Choquet simplex, that is, every
invariant measure $\mu$ admits a unique representation as an integral
average of the ergodic measures (the ergodic decomposition). Thus,
the system $(X,T)$ gives rise to what we call an \emph{assignment},
a function $\Psi$ whose domain is a metric Choquet simplex $K$,
and codomain is the set of measure-preserving systems identified up
to isomorphism. Every such assignment is determined by its restriction
to the set $\ex K$ of extreme points of $K$; the restriction assumes
only ergodic measure-preserving systems and the entire assignment
can be reconstructed from $\Psi|_{\ex K}$ according to the ergodic
decomposition. Two assignments, $\Psi$ on a simplex $\mathcal{P}$,
and $\Psi'$ on a simplex $\mathcal{P}'$ are said to be \textit{equivalent}
if there is an affine homeomorphism $\pi:\mathcal{P}\rightarrow\mathcal{P}'$
such that for every $p\in\mathcal{P}$, $\Psi(p)$ and
$\Psi(\pi(p))$ are isomorphic (as measure preserving systems).

In an attempt to understand the interplay between topological and
measurable dynamics, one encounters the following natural problem:
\begin{itemize}
	\item characterize the abstract assignments that can be realized in topological
dynamical systems.
\end{itemize}

At the moment the complete solution of this problem seems beyond our
reach. There should exist some ``continuity'' or at least ``measurability''
obstructions, but we have at our disposal no good topological or measurable
structure in the collection of classes of measure-preserving systems
modulo isomorphism. Nonetheless, we can produce a number of necessary
conditions by studying the behavior of some isomorphism invariants
with values in more friendly spaces. For instance, it is fairly intuitive,
that if we consider an isomorphism invariant in form of a real number
$r$ (for example the Kolmogorov-Sinai entropy), then $r(\Psi)$ should
be measurable on the simplex $K$. Indeed, in any topological system
the entropy function $h:\CM_{T}(X)\to[0,\infty]$ is not only measurable;
it is a nondecreasing limit of upper semicontinuous functions (i.e.,
of Young class LU), see \cite{DS03}.%
\footnote{ The entropy function is also affine. In \cite{DS03} it is shown
that there are no other entropy obstructions: every affine LU function
$h:\ex K\to[0,\infty]$ defined on any metrizable Choquet simplex
can be modeled as the entropy function in a topological system. %
} Since (except on some domains, e.g. discrete or countable) not every
nonnegative function is of class LU (suffice it to note that an LU
function is Borel measurable), the entropy obstruction is nonvoid;
it implies that not all possible assignments are admissible in topological
systems.%
\footnote{Interestingly, it has been proven (in \cite{KO06} for homeomorphisms
and independently in \cite{D8} for continuous maps) that if the simplex
$K$ has at most countably many extreme points then every assignment
$\Psi$ on $K$ assuming ergodic but not periodic ``values'' on
the extreme points (and extended to the rest of the simplex by averaging)
can be realized in a topological (even minimal) system. This is in
no collision with the entropy obstruction; on a countable set every
function is of class LU. Notice that this fact generalizes the celebrated
Jewett--Krieger Theorem, which can be viewed as a special case concerning
the one-point simplex. %
}

\smallskip{}
 Following the same lines of investigation, in this paper we will
seek an obstruction related to another real-valued (in fact integer-valued)
isomorphism invariant, namely the rank (as defined by Ornstein, Rudolph
and Weiss in \cite{ORW82}). Notice that rank distinguishes systems
of zero Kolmogorov-Sinai entropy, hence any obstruction that we find
is complementary to the entropy obstruction.

Although its definition does not require ergodicity, rank has been
studied mainly for ergodic systems. Any invariant measure on a standard space decomposes 
as an integral average of ergodic measures, hence it is important to understand how rank 
depends on convex combinations of mutually singular measures (and more general integral 
averages). In this aspect we will recall the result proven originally by J. King in \cite{K88}, 
that rank obeys a certain ``additive rule'' (it cannot be affine as it is an integer-valued 
function). Next, we will give the main result of this note, that just like the entropy function, 
the rank function is also of Young class LU. This gives a non-trivial restriction on possible 
assignments on metrizable Choquet simplices having an uncountable number of extreme points 
with ``values'' being measure-preserving systems with entropy zero.

\section{Preliminaries}

Let $(X,\CB,\mu)$ be a probability space. All subsets of $X$ considered
below are assumed to be measurable, and all partitions are (measurable
and) finite.
%\begin{defn} We say that a sequence of partitions $\set{\CP_{m}}_{m\ge1}$
%\emph{sequentially generates} if, given any set $A$ and $\eps>0$,
%there exists $N\in\mathbb{N}$ so that for all $m\geq N$, there exists
%a set $A_{m}$ being a union of elements of $\CP_{m}$, such that
%$\mu(A_{m}\triangle A)<\eps$.%
%\footnote{ Note that if a sequence of partitions $\set{\CP_{m}}_{m\ge1}$ sequentially
%generates, then it also \emph{generates}, i.e., $\bigvee_{m}\CP_{m}=\CB$
%up to measure $\mu$. Clearly, the reverse implication is not true. %
%}
%
%\end{defn}
Given a partition $\CP$ of $X$ and a subset $Y\subset X$, the symbol
$\CP|_{Y}$ denotes the partition $\{P\cap Y:P\in\CP\}$ of $Y$.

\begin{defn}\label{def:epsref} For partitions $\CP$ and $\CQ$ we will write $\CP\succ_{\eps}\CQ$ if there exists a set $Y_{\eps}$
of measure at least $1-\eps$ such that $\CP|_{Y_{\eps}}\succ\CQ|_{Y_{\eps}}$
(i.e., $\CP$ refines $\CQ$ relatively on $Y_{\eps}$). \end{defn}
%A fairly straightforward proof of the following statement is left
%to the reader as an exercise:
%\begin{lem}\label{lem:generates} A sequence of finite partitions
%$\set{\CP_{m}}_{m\ge1}$ sequentially generates if and only if, given
%any finite partition $\CQ$ and $\eps>0$, it holds that $\CP_{m}\succ_{\eps}\CQ$
%for all large enough $m$. \end{lem}
Throughout the rest of this section $X$ denotes a separable metric
space. Most of the definitions and statements hold for more general
topological spaces, but for the sake of simplicity, the theory is
not developed in wider generality. Exceptionally, the space $X$ from
this part of preliminaries will be later interpreted as either $\CM_{T}(Y)$
or $\ex\CM_{T}(Y)$ where $Y$ is a topological space.

\begin{defn} A function $f:X\rightarrow\mathbb{R}\cup\{-\infty,\infty\}$
is called \emph{upper semicontinuous} (u.s.c.) if for all $t\in\mathbb{R}$
the sets $\{x\in X:\, f(x)<t\}$ are open.%
\footnote{ One is accustomed to \emph{real-valued} u.s.c. functions, and such
are always bounded from above. In our setup a u.s.c. function is either
bounded from above or it assumes infinity as a value (necessarily
on a closed set). %
} \end{defn}

\begin{rem}\label{rem: usc}

The characteristic function of a closed set is upper semicontinuous
(see Chapter IV, $\S$6.2 of \cite{B98}). A function $f$ is upper
semicontinuous if and only if it is the pointwise limit of a nonincreasing
sequence of continuous functions from $X$ to $\mathbb{R}\cup\{-\infty,\infty\}$
(see p. 51 of \cite{R88}).

\end{rem}

The next definition originated in \cite{Y1911}:

\begin{defn} A function $f:X\to\mathbb{R}\cup\set{-\infty,\infty}$
is of Young class LU (an LU function for short) if it is the pointwise
limit of a nondecreasing sequence $f_{n}$ of upper semicontinuous
functions. \end{defn}
\noindent
The reader will easily verify that the class LU is closed under finite
sums, finite infima and countable suprema.

%\begin{lem}\label{lem:lu}
%A function $f:X\to \mathbb R\cup\{-\infty,\infty\}$ is of Young class LU if and only if, for any $t\in\mathbb R$,
%the set $\set{x:f(x)>t}$ is of type $F_\sigma$.
%\end{lem}
%\begin{proof}
%If $f$ is of class LU, then $f=\lim_{n\to\infty}\uparrow f_n$, where the $f_n$ are u.s.c. Therefore \[\set{x:f(x)\leq t}=\bigcap_{k}\set{x:\forall_n \ f_n(x)<t+\tfrac{1}{k}}=\bigcap_{n,k}\set{x:f_n(x)<t+\tfrac{1}{k}}\] and thus it is of type $G_\delta$, so its complement is of type $F_\sigma$.
%To prove the other implication, let $f_t = t\cdot 1_{\{x:f(x)>t\}}$. By assumption, the set $\{x:f(x)>t\}$ is of type $F_\sigma$,
%so its characteristic function is of Young class LU (being a nondecreasing limit of characteristic functions of closed sets, which
%are u.s.c). Now it remains to note that $f = \sup f_q$ with $q$ ranging over the rationals (and use the fact that the class LU is
%closed under countable suprema).
%\end{proof}

\section{Rank of measure preserving systems}

Let $(X,\CB,\mu,T)$ be a measure-preserving dynamical system on a
standard probability space\footnote{Recall that a standard probability space is isomorphic to a union
between a countable atomic measure and the unit interval equipped
with the Lebesgue measure defined on the completion of the Borel
$\sigma$-algebra.}  ($T$ is not assumed to be invertible).
We will be studying the properties of an isomorphism invariant called
rank. The notion was defined in {[}ORW82{]}. In fact the class of
rank-one systems was known much earlier (since the 1960's) but the
term ``rank'' was not yet used. The reader is referred also to the
survey by Ferenczi (\cite{F97}) for more details. Below we provide
the definition of rank preceded by an auxiliary notion.%The definition
%we provide below (Definition~\ref{def:rank}) is not identical, yet
%equivalent (in standard spaces) to the one found in the aforementioned
%sources.

\begin{defn} Let $(X,\CB,\mu,T)$ be a measure-preserving dynamical system. Let $n_{1},\ldots,n_{k}\in\Nat$. We call $\CT_{1}=\{B_{1},TB_{1},\ldots,T^{n_{1}-1}B_{1}\},$
$\CT_{2}=\{B_{2},TB_{2},\ldots,T^{n_{2}-1}B_{2}\},\dots,\CT_{k}=\{B_{k},TB_{k},\ldots,T^{n_{k}-1}B_{k}\}$
 $k$ \emph{disjoint measurable towers}, if all the sets $T^{i}B_{l}\subset X$
appearing in the towers are measurable and
\[
T^{i}B_{l}\cap T^{i'}B_{l'}=\emptyset\text{ for }(l,i)\neq(l',i').
\]
Introduce the index set $I=I(\CT_{1},\ldots,\CT_{k})\triangleq\bigcup_{l=1}^{k}(\{l\}\times\{0,\dots,n_{l}-1\})$.
A \emph{k-tower partition} $\CP$ associated with $\CT_{1},\CT_{2},\ldots,\CT_{k}$
is the partition of $X$ consisting of the sets $T^{i}B_{l}$ with $(l,i)\in I$
and $R\triangleq X\setminus\bigcup_{(l,i)\in I}T^{i}B_{l}$. The sets
$B_{l}$, $l=1,\ldots,k$ will be referred to as the \emph{bases of
the towers}, the sets $T^{i}B_{l}$, $(l,i)\in I$ (including the
bases) \emph{the level sets}, and $R$ will be referred to as the
\emph{the remainder}. One writes
\[
\CP=\{T^{i}B_{l},R\}_{(l,i)\in I}
\]
We will skip the ranges of the indices (i.e. the notation $(l,i)\in I$)
whenever possible. \end{defn}

\begin{defn}\label{def:rank} Let $k\in\Nat$. One says that $\rank(\mu)\le k$
if for any finite partition $\CQ$ of $X$ and any $\eps>0$ there exists a $k$-tower partition $\CP$,
such $\CP\succ_\eps \CQ$ and the measure of the remainder of $\CP$ is less than $\eps$.%
\footnote{ The requirement on the remainder is meant to assure that every fixpoint of positive
measure eventually requires a separate tower. Without this condition such a fixpoint could be included
in all remainders and the resulting ``rank'' would be lowered by~1. The condition is automatically
fulfilled for nonatomic measures, and can be dropped in any systems possessing no fixpoints of positive
measure. It seems that this detail was overlooked in \cite{K88}.%
} Otherwise we say that $\rank(\mu)>k$. $\rank(\mu)=k$ means that
$\rank(\mu)\le k$ and $\rank(\mu)>k-1$. Finally, $\rank(\mu)=\infty$
if $\rank(\mu)>k$ for any natural $k$. $\rank(\mu)$ is referred
to as the $\emph{rank}$ of $(X,\CB,\mu,T)$. \end{defn}
\noindent
Notice that the definition does not require $\mu$ to be ergodic.
If $\mu$ is not ergodic then, since $(X,\CB,\mu)$ is standard, we
have the following formula (called the \emph{ergodic decomposition}):
\[
\mu=\int\nu_{y}\, d\xi_{\mu}(y),
\]
where $\xi_{\mu}$ is the projection of $\mu$ onto the
$\sigma$-algebra $\Sigma$ of invariant sets, and the measures $\nu_{y}$
are ergodic and supported by the atoms $y$ of $\Sigma$. Distinct
ergodic measures are mutually singular.

\begin{rem}\label{rem:invertibility} It is well known (for ergodic
systems, but easily extends to nonergodic systems as well due to the additivity of rank) that finite
rank implies zero Kolmogorov--Sinai entropy (see proposition 1.8 of \cite{K88}),
which further implies that the transformation $T$ is invertible $\mu$-almost
everywhere. We will use this fact several times. \end{rem}

We now prove the following ``additive rule'', which is a simple generalization of Lemma 1.2 of \cite{K88}:

\begin{thm}\label{thm:additive} The rank function satisfies
\[
\rank(\mu)=\sum_{\nu\in\supp(\xi_{\mu})}\rank(\nu).
\]
\end{thm} In other words, $\rank(\mu)$ can be finite only for measures
which are convex combinations of finitely many ergodic measures, and
then it equals the sum of ranks of the ergodic components (of course
only those with strictly positive coefficients). Other measures have
infinite rank.

\begin{proof} If $\mu$ is a finite combination of ergodic measures,
this follows directly from lemma \ref{lem:convex} below. Every
other measure can be represented as a convex combination of arbitrarily
(still finitely) many mutually singular measures (each of rank at
least 1 -- there is no rank zero), so, by the same theorem its rank
is infinite. Clearly, the sum on the right is, for such measures,
also infinite, so the equality holds. \end{proof}

The following lemma is a reformulation of Lemma 1.2 of \cite{K88}\footnote{In the statement of this lemma it is assumed that the measures forming the convex combination are ergodic, but this property is not actually used in the proof.}:
\begin{lem}\label{lem:convex} Let $(X,\CB,\mu,T)$ be a dynamical
system. Suppose that $\mu=p\mu_{1}+q\mu_{2}$, where $p\in(0,1),q=1-p$,
$\mu_{1}$ and $\mu_{2}$ are mutually singular $T$-invariant measures
on $\CB$. Then:
\[
\rank(\mu)=\rank(\mu_{1})+\rank(\mu_{2}).
\]
\end{lem}

\section{Rank in topological systems}

\bigskip{}
 Throughout, by a \emph{topological dynamical system} (t.d.s.) we will
mean a pair $(X,T)$, where $X$ is a metric space (with the metric
denoted by $d$) and $T:X\to X$ is continuous (not necessarily invertible). We will denote by
$\CM_{T}(X)$ the set of all $T$-invariant Borel probability measures
on $X$ and by $\ex\CM_{T}(X)\subset\CM_{T}(X)$ the set of ergodic
measures. It is well known that both sets are nonempty and the former
set equals the simplex whose extreme points constitute the latter
set. For $\mu\in\CM_{T}(X)$ we obtain a measure-theoretic dynamical
system $(X,\CB_{\mu},\mu,T)$, where $\CB_{\mu}$ denotes the $\sigma$-algebra
of Borel sets completed with respect to $\mu$. Note that $(X,\CB_{\mu},\mu)$
is a standard probability space. We are now going to investigate the
\emph{rank function} $\rank:\CM_{T}(X)\to\Nat\cup\{\infty\}$, $\mu\mapsto\rank(\mu)$
(computed in the system $(X,\CB_{\mu},\mu,T)$).
We would like to use the ``additive rule'' of Lemma \ref{lem:convex}
(also in the form of Theorem \ref{thm:additive}) in this context.
Notice however that for two distinct $\mu_{1},\mu_{2}\in\CM_{T}(X)$,
it may happen that $\CB_{\mu_{1}}\neq\CB_{\mu_{2}}$, so technically
the assumption of this theorem may not hold. The following remark
explains why still we can use Lemma \ref{lem:convex}.

\begin{rem} For a t.d.s. rank does not depend on taking the completion.
Indeed if the rank computed for the completed Borel $\sigma$-algebra
is infinite then clearly it is also infinite for the Borel $\sigma$-algebra.
Otherwise the map is invertible mod $\mu$ and then any tower is equal
mod $\mu$ to a tower with Borel level sets. It suffices to replace
the base by its subset (of equal measure) of type $F_{\sigma}$ (a
countable union of closed sets), and note that type $F_{\sigma}$
is preserved by forward images of continuous maps on compact spaces.
By invertibility, the forward images of the discarded part of the
base have measure zero and can be discarded from the level sets. \end{rem}
\noindent
As a consequence of the previous remark we will write $(X,\CB,\mu,T)$ instead of $(X,\CB_{\mu},\mu,T)$ when no confusion arises. The main goal of this section is proving the following theorem:

\begin{thm}\label{thm:main} In any topological dynamical system
$(X,T)$ the rank function $\rank:\CM_{T}(X)\to\Nat\cup\{\infty\}$
is of Young class LU. \end{thm}

In order to prove the theorem we define an approximate notion of rank
applicable to measure-theoretic dynamical systems $(X,\CB,\mu,T)$
arising from a t.d.s. $(X,T)$ (where $X$ is a metric space).

\begin{defn}\label{defn:erank} Let $(X,T)$ be a t.d.s. and let $k\in\Nat$. One says that $\rank_{\eps}(\mu)\leq k$
if there exist a measurable $k$-tower partition $\CP$ of $(X,\CB,\mu,T)$ whose
remainder satisfies $\mu(R)<\eps$, and a measurable set $X_{\eps}$
such that $\mu(X_{\eps})>1-\eps$ and all elements of $\CP|_{X_{\eps}}$
have diameters smaller than $\eps$. Otherwise we say that $\rank_{\eps}(\mu)>k$.
$\rank_{\eps}(\mu)=k$ if $\rank_{\eps}(\mu)\le k$ and $\rank_{\eps}(\mu)>k-1$.
$\rank_{\eps}(\mu)=\infty$ if $\rank_{\eps}(\mu)>k$ for all natural
numbers $k$. $\rank_{\eps}(\mu)$ is referred to as the \emph{$\eps$-rank}
of $(X,\CB,\mu,T)$. \end{defn}

\begin{rem} The definition does not imply existence of a $k$-tower
partition whose all sets except the remainder (of small measure) have
diameters smaller than $\eps$. To achieve small diameters, we may
need to discard large parts (in relative measure) from some level
sets scattered along the towers. Although the discarded parts have
small measure, this may destroy the tower structure on a set of large
measure. \end{rem}

Observe that if $\eps<\eps'$ then $\rank_{\eps}\ge\rank_{\eps'}$.

\begin{lem}\label{lem:ranklim} Let $(X,T)$ be a topological dynamical
system. Let $\mu\in\CM_{T}(X)$. Then $\rank(\mu)\leq k$ if and only
if $\rank_{\eps}(\mu)\leq k$ for all $\eps>0$. In other words (by
monotonicity)
\[
\rank(\mu)=\lim_{m}\uparrow\rank_{\eps_{m}}(\mu),
\]
whenever $\eps_{m}\searrow0$. \end{lem}

\begin{proof} Assume that $\rank_{\eps}(\mu)\leq k$ for all $\eps>0$. Fix a partition $\CQ=\set{Q_1,\ldots,Q_n}$ and $\eps>0$. The measure $\mu$ (being a Borel measure on a metric space) is regular, therefore there exist compact sets $K_i\subset Q_i$ such that $\mu(Q_i\setminus K_i)<\frac{\eps}{4n}$, $i=1,\ldots,n$. $\mu$ is also continuous from above, so there exists a $\delta>0$ such that $\mu(K_i^\delta)<\mu(K_i)+\frac{\eps}{4n}$, where $K_i^\delta$ denotes the $\delta$-neighborhood of $K_i$. We can also assume $\delta$ to be so small that the sets $K_i^\delta$ are disjoint.

Let $\eta<\min\set{\delta,\frac{\eps}{4n}}$ and let $\CP_\eta$ and $X_\eta$ be as in the definition of $\rank_\eta(\mu)$. Set
\[C_i=\bigcup\set{P\cap X_\eta:P\in \CP_\eta, (P\cap X_\eta)\cap K_i\neq \emptyset}.\]
Observe that $K_i\cap X_\eta\subset C_i\subset K_i^\delta$. Since $\mu(X_\eta)>1-\eta>1-\frac{\eps}{4n}$, we see that $\mu(C_i\sd K_i)<\frac{2\eps}{4n}$. The triangle inequality for the metric $\mu(\cdot \sd \cdot)$ yields the estimate $\mu(C_i \sd Q_i)<\frac{3\eps}{4n}$. The set
\[A_i=\bigcup\set{P\in \CP_\eta: (P\cap X_\eta)\cap K_i\neq \emptyset}\]
is a sum of elements of $\CP_\eta$, we also know that $A_i\cap X_\eta=C_i$. Therefore $A_i$ differs from $C_i$ (in measure) by at most $\frac{\eps}{4n}$, and thus $\mu(A_i\sd Q_i)<\frac{\eps}{n}$. Furthermore, the sets $A_i$ are disjoint, which implies that $\CP_\eta\succ_\eps\CQ$. The remainder of $\CP_\eta$ has measure less than $\frac{\eps}{4n}<\eps$, so we conclude that $\rank(\mu)\leq k$.

The reversed implication will be first handled for nonatomic measures
$\mu$. Let $\eps>0$. Since $\mu$ is nonatomic, there exists
a partition $\CQ=\{Q_{1},\dots,Q_{l}\}$ of $X$ whose elements have
measures smaller than $\eps$ and, moreover, diameters smaller than
$\eps$. Since $\rank(\mu)\leq k$, we can find a $k$-tower partition $\CP$
such that $\CP \succ_\eps\CQ$ and the remainder of $\CP$ has measure less than $\eps$.
By definition this determines a set $Y_{\eps}$ with $\mu(Y_{\eps})>1-\eps$
and such that $\CP|_{Y_{\eps}}\succ\CQ|_{Y_{\eps}}$. Thus the
elements of $\CP|_{Y_{\eps}}$ have diameters smaller than $\eps$.
According to Definition \ref{defn:erank}, we have shown that $\rank_{\eps}(\mu)\le k$.

If $\mu$ has atoms then $\mu=p\mu'+\sum_{i=1}^{n}q_{i}\mu_{i}$ ($0\le p<1$,
$p+\sum_{i=1}^{n}q_{i}=1$, $n\in\mathbb{N}$), where
$\mu'$ is nonatomic and each $\mu_{i}$ is a periodic measure supported
by an individual periodic orbit. Since the measures $\mu',\mu_{1},\mu_{2},\ldots,\mu_{n}$
are mutually singular and the rank of any measure is at least $1$,
we conclude (using the ``additive rule'' of Lemma \ref{lem:convex})
that $k\ge k'+n$, where $k'=\rank(\mu')$ (this also explains why
the number $n$ of periodic orbits must be finite). Given $\eps>0$,
let $\CP$ be the $(k'\!+\! n)$-tower partition consisting of the
$n$ periodic orbits (each viewed as a tower with singleton level
sets) and a $k'$-tower partition of the rest of the space, satisfying
the conditions of Definition \ref{defn:erank} for $\mu'$ (with some
set $X'_{\eps}$), whose existence is established in the preceding
paragraph. It is clear that $\CP$ fulfills the requirements of Definition
\ref{defn:erank} showing that $\rank_{\eps}(\mu)\le k'+n\le k$;
the set $X_{\eps}$ equals the union of $X'_{\eps}$ and the periodic
orbits (then $\mu(X_{\eps})>1-p\eps>1-\eps$). \end{proof}

The following theorem is the key observation of this work. For easier
proof we assume invertibility of the transformation. Subsequently we will use it to prove Theorem \ref{thm:main}
(which holds in the general case, i.e. also for non-invertible transformations).

\begin{thm}\label{thm:rank} Let $(X,T)$ be an \emph{invertible}
(i.e., in which $T$ is a homeomorphism) topological dynamical system.
For any $\eps>0$ the function $\rank_{\eps}(\mu):\CM_{T}(X)\rightarrow\mathbb{N}\cup\{\infty\}$
is upper semicontinuous. \end{thm} \begin{proof} We need to show
that for each $t\in\mathbb{R}$, $\rank_{\eps}(\mu)<t$ holds on an
open set of invariant measures. Since $\rank_{\eps}$ assumes only
natural values (or $\infty$) this set of measures is nonempty only
for $t>1$ and then the condition $\rank_{\eps}(\mu)<t$ can be equivalently
replaced by $\rank_{\eps}(\mu)\le k$ for some $k\in\mathbb{N}$.

Assume $\rank_{\eps}(\mu)\le k$. This means there exists a measurable
$k$-tower partition $\CP=\{T^{i}B_{l},R\}_{(l,i)\in I}$ %$(I=\bigcup_{t=1}^{k}\{t\}\times\{0,1,\dots,n_{t}-1\})$
with $\mu(R)<\eps$ and a set $X_{\eps}$ with $\mu(X_{\eps})>1-\eps$
such that $\CP|_{X_{\eps}}$ consists of sets with diameters smaller
than $\eps$. We will now explain why we can assume that all level
sets of the towers are closed. Choose a positive number $\xi$ such
that $\mu(R)+\xi<\eps$ and $\mu(X_{\eps})-\xi>1-\eps$. By regularity
we can find closed subsets sets $B'_{l}\subset B_{l}$ so that $\mu(B_{l}\setminus B'_{l})<\frac{\xi}{kn_{l}}$.
Then, for all pairs $(l,i)\in I$ the images $T^{i}B'_{l}$ are closed,
contained in $T^{i}B_{l}$ and
\[
\mu(T^{i}B_{l}\setminus T^{i}B'_{l})<\tfrac{\xi}{kn_{l}}
\]
(here we use the assumption that $T$ is invertible). Let $\CP'=\{T^{i}B'_{l},R'\}_{(l,i)\in I}$
be the $k$-tower partition associated with the new (smaller) bases
$B_{l}'$ and a new (larger) remainder set $R'$. The difference $R'\setminus R$
equals the union of the parts discarded from the level sets, so its
measure is smaller than $\xi$. Thus $\mu(R')\le\mu(R)+\xi<\eps$.
Similarly, the set $X'_{\eps}=X_{\eps}\setminus(R'\setminus R)$ has
measure larger than $\mu(X_{\eps})-\xi>1-\eps$. Because $X'_{\eps}$
differs from $X_{\eps}$ only within the new remainder, the partition
$\CP'|_{X'_{\eps}}$ consists of the sets $T^{i}B'_{l}\cap X'_{\eps}=T^{i}B'_{l}\cap X_{\eps}\subset T^{i}B_{l}\cap X_{\eps}$
(which have diameters smaller than $\eps$) and $R'\cap X'_{\eps}=R\cap X_{\eps}$
(also of diameter smaller than $\eps$).

From now on we assume that the original $k$-tower partition $\CP=\{T^{i}B_{l},R\}$
has closed level sets. We are going to modify the tower and the set
$X_{\eps}$ once more, so that $R$ becomes closed and $X_{\eps}$
open. Once again, choose a positive $\xi$ such that $\mu(X_{\eps})-\xi>1-\eps$
(this time the other condition, involving the remainder, will not be needed).
Define \[R^{-\delta}=\set{x\in R:d(x,X\setminus R)\geq \delta}.\]
Find $\delta$ such that $\mu(R\setminus R^{-\delta})<\xi$ and let $\alpha$ denote
a positive number smaller than half of the smallest distance between
two distinct closed sets from the family $\{T^{i}B_{l},R^{-\delta}\}$
(clearly $\alpha\le\frac{\delta}{2}$). Let $\beta>0$ be so small
that $d(x,y)<\beta\implies d(T^{i}x,T^{i}y)<\alpha$ for all $0\le i<\max\{n_{1},\dots, n_{k}\}$
(clearly $\beta\le\alpha$). Define $B'_{l}=B_{l}^{\beta}$ (i.e.,
the open $\beta$-neighborhood around $B_{l}$).
For every pair $(l,i)\in I$ we have $T^{i}B'_{l}\subset(T^{i}B_{l})^{\alpha}$
which implies that the sets $T^{i}B'_{l}$ are pairwise disjoint,
hence form a new $k$-tower partition $\CP'$ with a new smaller and
closed remainder $R'$.

Let $\gamma$ be such that $(T^{i}B_{l})^{\gamma}\subset T^{i}B'_{l}$
for all pairs $(l,i)\in I$ (here we use again that $T$ is a homeomorphism,
so the new level sets $T^{i}B'_{l}$ are all open neighborhoods of
the old closed level sets $T^{i}B{}_{l}$). Clearly, $\gamma\le\beta\le\alpha$.
We can choose $\gamma$ also smaller than half of the difference between
$\eps$ and the largest diameter of an element of $\CP|_{X_{\eps}}$.
We can now define the modified \emph{open} version of $X_{\eps}$:
\[
X'_{\eps}=\bigcup_{(l,i)\in I}(T^{i}B_{l}\cap X_{\eps})^{\gamma}\cup(R^{-\delta}\cap X_{\eps})^{\gamma}.
\]
Notice that $X'_{\eps}$ contains $X_{\eps}$ except its part contained
in $R\setminus R^{-\delta}$ (this is seen even if we disregard the
$\gamma$-neighborhoods). So the measure of $X'_{\eps}$ has dropped
by at most $\xi$ and thus is still larger than $1-\eps$. Since $(T^{i}B_{l}\cap X_{\eps})^{\gamma}\subset(T^{i}B_{l})^{\gamma}\subset T^{i}B'_{l}\subset(T^{i}B_{l})^{\alpha}$
for all pairs $(l,i)$ and $(R^{-\delta}\cap X_{\eps})^{\gamma}\subset(R^{-\delta})^{\alpha}$,
the items of the union defining $X'_{\eps}$ are pairwise disjoint
(as $\{(T^{i}B_{l})^{\alpha},R^{\alpha}\}$ are pairwise disjoint), and
each new level set $T^{i}B'_{l}$ intersects only one of them, namely
$(T^{i}B_{l}\cap X_{\eps})^{\gamma}$. Moreover, as  $T^{i}B'_{l}\subset(T^{i}B_{l})^{\alpha}$ we have
in fact the following equality :
\[
T^{i}B'_{l}\cap X'_{\eps}=(T^{i}B_{l}\cap X_{\eps})^{\gamma}.
\]
This implies that the last item $(R^{-\delta}\cap X_{\eps})^{\gamma}$
equals the intersection of $X'_{\eps}$ with the remainder $R'$ of
the new tower. We have just proven that the items of the union defining
$X'_{\eps}$ correspond to the elements of the partition $\CP'|_{X'_{\eps}}$.
As $\gamma$ is smaller than half of the difference between $\eps$
and the largest diameter of an element of $\CP|_{X_{\eps}}$ (and
since $R^{-\delta}\subset R$), the diameters of all these items are
smaller than $\eps$.

To summarize, we have shown that if $\rank_{\eps}(\mu)\le k$ then
we can arrange the partition $\CP$ with a closed reminder $R$, so
that the conditions in Definition \ref{defn:erank} are fulfilled
with an open set $X_{\eps}$. Because the measure of a closed set
is an upper semicontinuous function of the measure (see Remark \ref{rem: usc}),
we have $\mu'(R)<\eps$ and $\mu'(X\setminus X_{\eps})<\eps$ (i.e.,
$\mu'(X_{\eps})>1-\eps$) on an open set of measures (containing $\mu$).
The same partition $\CP$ and the same set $X_{\eps}$ now give that
$\rank_{\eps}(\mu')\le k$ for all these measures, concluding the
proof. \end{proof}

We can now prove the main result of this paper.

\begin{proof}{[}Proof of Theorem \ref{thm:main}{]} If $T$ is a
homeomorphism, the result is a direct consequence of the preceding
Theorem \ref{thm:rank} and Lemma \ref{lem:ranklim}. For noninvertible
maps we first embed $(X,T)$ in another t.d.s. $(X',T')$ such that
$T'$ is surjective on $X'$ (see the paragraph after Definition 6.8.10
on p. 189 of \cite{D11}). Next we construct an extension $(X'',T'')\rightarrow(X',T')$
so that $T''$ is a homeomorphism and every $T'$-invariant measure $\mu'\in M_{T'}(X')$ lifts to a unique $T''$-invariant measure $\mu''\in M_{T''}(X'')$ such that the measure-theoretic system $(X'',\CB_{X''},\mu'',T'')$
is isomorphic to the measure-theoretic natural extension of $(X',\CB_{X'},\mu',T')$.
Moreover, the correspondence $\mu'\mapsto\mu''$ is a homeomorphism
between $\CM_{T'}(X')$ and $\CM_{T''}(X'')$ (the details of this
construction, the so called \textit{topological natural extension},
can be found in \cite{D11} pages 189-190 and pages 111-112). We will
argue that $\rank(\mu')=\rank(\mu'')$. If $\mu'$ has entropy zero
then $T'$ is invertible modulo $\mu'$, and then the system $(X',\CB_{X'},\mu',T')$
is isomorphic to its own natural extension, and thus to $(X'',\CB_{X''},\mu'',T'')$,
which obviously implies the desired equality of the ranks. Otherwise
both $\mu'$ and $\mu''$ have nonzero entropy hence infinite rank.
Since we already know that the rank function is of class LU on $\CM_{T''}(X'')$,
it follows that it is of the same class on $\CM_{T'}(X')$ and hence,
by restriction, on $\CM_{T}(X)$. \end{proof}

\section{Open Questions}

We have obtained that for any topological dynamical system, the rank
function defined on $\CM_{T}(X)$ is of Young class LU and obeys the
``additive rule'' of Theorem \ref{thm:additive}. In particular,
this function is completely determined by its restriction to $\ex\CM_{T}(X)$,
which obviously is also of class LU. Two natural question arise:
\begin{enumerate}
\item Given a metrizable Choquet simplex $K$ and an LU function on $\ex K$,
is its extension to all of $K$ by the ``additive rule'' automatically
of class LU?
\item Are these the only ``rank obstructions''? I.e., given an LU function
on a metrizable Choquet simplex $r:K\to\mathbb{N}\cup\{\infty\}$
which obeys the ``additive rule'', does there exist a topological
dynamical system (perhaps minimal) realizing $r$ as the rank function
on the simplex of invariant measures?
\end{enumerate}
\bibliographystyle{alpha}
\bibliography{universal_bib}

\def\cprime{$'$}
\begin{thebibliography}{ORW82}

\bibitem[Bou98]{B98}
Nicolas Bourbaki.
\newblock {\em General topology. {C}hapters 1--4}.
\newblock Elements of Mathematics (Berlin). Springer-Verlag, Berlin, 1998.
\newblock Translated from the French, Reprint of the 1989 English translation.

\bibitem[Dow08]{D8}
Tomasz Downarowicz.
\newblock Faces of simplexes of invariant measures.
\newblock {\em Israel J. Math.}, 165:189--210, 2008.

\bibitem[Dow11]{D11}
Tomasz Downarowicz.
\newblock {\em Entropy in dynamical systems}, volume~18 of {\em New
  Mathematical Monographs}.
\newblock Cambridge University Press, Cambridge, 2011.

\bibitem[DS03]{DS03}
Tomasz Downarowicz and Jacek Serafin.
\newblock Possible entropy functions.
\newblock {\em Israel J. Math.}, 135:221--250, 2003.

\bibitem[Fer97]{F97}
S{\'e}bastien Ferenczi.
\newblock Systems of finite rank.
\newblock {\em Colloq. Math.}, 73(1):35--65, 1997.

\bibitem[Kin88]{K88}
Jonathan~L. King.
\newblock Joining-rank and the structure of finite rank mixing transformations.
\newblock {\em J. Analyse Math.}, 51:182--227, 1988.

\bibitem[KO06]{KO06}
Isaac Kornfeld and Nicholas Ormes.
\newblock Topological realizations of families of ergodic automorphisms,
  multitowers and orbit equivalence.
\newblock {\em Israel J. Math.}, 155:335--357, 2006.

\bibitem[ORW82]{ORW82}
Donald~S. Ornstein, Daniel~J. Rudolph, and Benjamin Weiss.
\newblock Equivalence of measure preserving transformations.
\newblock {\em Mem. Amer. Math. Soc.}, 37(262):xii+116, 1982.

\bibitem[Roy88]{R88}
H.~L. Royden.
\newblock {\em Real analysis}.
\newblock Macmillan Publishing Company, New York, third edition, 1988.

\bibitem[You11]{Y1911}
W.~H. Young.
\newblock On a new method in the theory of integration.
\newblock {\em Proceedings of the London Mathematical Society}, s2-9(1):15--50,
  1911.

\end{thebibliography}

\bigskip\noindent
\textsc{Tomasz Downarowicz, Institute of Mathematics and Computer Science, Wroclaw University of Technology, Wybrze\.ze Wyspia\'nskiego 27, 50-370 Wroc\l aw, Poland.}

\noindent
\emph{E-mail address:} \texttt{downar@pwr.wroc.pl}
\vskip 0.5 cm
\noindent
\textsc{Yonatan Gutman, Institute of Mathematics, Polish Academy of Sciences, ul. \'{S}niadeckich 8,00-956 Warszawa, Poland.}

\noindent
\emph{E-mail address:} \texttt{y.gutman@impan.pl}
\vskip 0.5 cm
\noindent
\textsc{Dawid Huczek, Institute of Mathematics and Computer Science, Wroclaw University of Technology, Wybrze\.ze Wyspia\'nskiego 27, 50-370 Wroc\l aw, Poland.}

\noindent
\emph{E-mail address:} \texttt{dawid.huczek@pwr.wroc.pl}

\end{document}